\documentclass[a4paper,10pt]{amsart}
\usepackage{amsmath,amsfonts,amssymb}
\usepackage[all]{xy}
\usepackage{hyperref}
\usepackage[latin1]{inputenc}
\usepackage{mathrsfs}

\newcommand{\R}{\mathbb R}

\newcommand{\ii}í

\newtheorem{theorem}{Theorem}[section]
\newtheorem{corollary}[theorem]{Corollary}
\newtheorem{lemma}[theorem]{Lemma}
\newtheorem{proposition}[theorem]{Proposition}

\theoremstyle{definition}
\newtheorem{definition}[theorem]{Definition}
\newtheorem{remark}[theorem]{Remark}
\numberwithin{equation}{section}

\begin{document}

\title[The Implicit Function Theorem] {The Implicit Function Theorem for continuous functions}
\author[C. Biasi, C. Gutierrez and E. L. dos Santos]
{Carlos Biasi, Carlos Gutierrez and Edivaldo L. dos Santos}

\address{Carlos Biasi and Carlos Gutierrez}

\address{Universidade de S\~ao Paulo,
Departamento de Matemática-ICMC, Caixa Postal $_{}$ \hspace{.15cm}
668, 13560-970, S\~ao Carlos SP, Brazil}
\email{biasi@icmc.usp.br}\email{gutp@icmc.usp.br}

\address{Edivaldo L. dos Santos}

\address{Universidade Federal de São Carlos, Departamento de
Matemática  Caixa $_{}$ \hspace{.35cm} Postal 676, 13565-905, S\~ao
Carlos SP, Brazil} \email{edivaldo@dm.ufscar.br}

\begin{abstract}
{In the present paper we obtain a new homological version of the
implicit function theorem and some versions of the Darboux theorem.
Such results are proved for continuous maps on topological
manifolds. As a consequence, some versions of these classic theorems
are proved when we consider differenciable (not necessarily $C^{1}$)
maps.}

\end{abstract}
\thanks{}
 \maketitle
\maketitle {\section{ Introduction }

This paper deals with a generalization of the classical Implicit
Function Theorem. In this respect, C. Biasi and E. L. dos Santos
proved {\it a homological version of the implicit function theorem}
for continuous functions on general topological spaces which has
interesting applications in the theory of topological groups. More
specifically, Theorem 2.1 of \cite{Ed} states that: {``If $X,Y,Z$
are Hausdorff spaces, with $X$ locally connected, $Y$ locally
compact and $f:X\times Y\to Z$ such that
$(f_{x_{0}})_{*}=(f|_{(x_{0}\times Y)})_{*} :H_{n}(Y,Y-y_{0})\to
H_{n}(Z,Z-z_{0})$ is a nontrivial homomorphism, for some natural
$n>0$, where $\{y_{0}\}=(f_{x_{0}})^{-1}(\{z_0 \})$, then there
exists an open neighborhood $V$ of $x_0$ and  a function $g:V\subset
X\to Y$ satisfying the equation $f(x,g(x))=z_{0}$, for each $x\in
V$. Moreover, $g$ is continuous at $x_0$".} This result establishes
the existence of an implicit function $g$, however,  the continuity
of such function is guaranteed only at the point $x_{0}$. In this
paper, under little stronger assumptions, we can ensure the
continuity of $g$ at a neighborhood of $x_0$. More precisely,

\medskip

\noindent {\bf{Theorem \ref{IFT}}} \emph{Let $X$ be a locally
pathwise connected Hausdorff space and let $Y,Z$ be oriented
connected topological manifolds of dimension $n$. Let $f:X \times \,
Y \,\to\, Z$ be a continuous map such that, for all $x \in X$, the
map $f_x:Y \to Z$ given by $f_x(y)=f(x,y)$ is open and discrete.
Suppose that for some $(x_0,y_0) \in X \times \, Y$,
$|\deg(f_{x_0},y_0)|=1$. Then there exists a neighborhood $V$ of
$x_0$ in $X$ and a continuous function $g:V \to Y$ such that
$f(x,g(x))=f(x_0,y_0)$, for all $x \in V$.}

\medskip

As a consequence we obtain

\medskip

\noindent {\bf Corollary \ref{coro-dif}} \emph{Let $X$ be a locally
pathwise connected Hausdorff space. Let $U\subset \mathbb{R}^{n}$ be
an open subset of $\mathbb{R}^{n}$ and let $f:X\times U\to
\mathbb{R}^{n}$ be a continuous map. Suppose that $f_{x}:U\to
\mathbb{R}^{n}$ is a differentiable (not necessarily $C^{1}$) map
without critical points, for each $x\in X$. Then there exist a
neighborhood $V$ of $x_{0}$ and a continuous function $g:V\to U$
such that $f(x,g(x))=z_{0}$, for each $x\in V$.}

\medskip

\noindent{\bf Corollary \ref{coro-eq-dif}} \emph{Let $I \times V
\times W$ be an open neighborhood of $(t_0,x_0,y_0)$ in $\R \times
\R^n \times \R^m$ and let $F:(I \times V \times W, (t_0,x_0,y_0))
\longmapsto (\R^m,0)$ be a continuous map. Suppose that, for all
$(t,x) \in I \times V$, the map $y \in W \longrightarrow F(t,x,y)$
is differentiable (but not necessarily $C^1$) and without critical
points. Then the differential equation
$$F(t,x,x')=0, \quad x(t_0)=x_0, \quad x'(t_0)=y_0$$
has a solution in some interval
$(t_0-\varepsilon,t_0+\varepsilon)$.}

\medskip

We also prove the following versions of Darboux Theorem

\medskip

\noindent {\bf Theorem \ref{HDT}} \emph{Let $M$ and $N$ be oriented
connected topological manifolds of dimension $n$ and let $f:M\to N$
be a continuous map. Suppose that there exist $x_{0}$ and $x_{1}$ in
$M$ such that $\deg(f,x_{0})<0$ and $\deg(f,x_{1})>0$. Then there
exists $x_{2}$ in $M$ such that $\deg(f,x_{2})=0$}

\medskip

\noindent{\bf Corollary \ref{DDT}} \emph{Let $M$ and $N$ be oriented
connected topological manifolds of dimension $n$ and let $f:M\to N$
be a differentiable map. Suppose that there exist $x_{0}$ and
$x_{1}$ in $M$ such that $\det[f^{\prime}(x_{0})]<0$ and
$\det[f^{\prime}(x_{1})]>0$. Then $f$ has a critical point.}

\medskip

A fundamental step to establish the versions of the Implicit
Function Theorem and Darboux Theorem is Lemma \ref{keylemma} (Key
Lemma).

\section{Preliminares}

 In this section we introduce  some basic notions, notations and
results that will be used throughout this paper. All considered
singular homology groups will always have coeffi\-cients in
$\mathbb{Z}$. By {\it dimension} we understand the usual topological
dimension  in the sense of \cite{hw}.


In the following definitions, $X$ and $Y$ will be oriented connected
topological manifolds of dimension $n\geq 1$ and  $f:X\to Y$ will be
a continuous map. The definition of orientation for topological
manifolds can be found, for instance, in \cite{dold}.

\begin{definition}\rm A map $f:X\to Y$ is said
to be {\it discrete} at a point $x_{0}$, if there exists a
neighborhood $V$ of $x_{0}$ such that $f(x)\not=f(x_{0})$, for any
$x\in V-x_{0}$, that is, $f^{-1}(f(x_{0}))\cap (V-x_{0})=\emptyset$.
\end{definition}

\begin{definition} \rm Let $y\in Y$ such that $L_{y}=f^{-1}(y)$
is a compact subset of X. Let $\alpha_{L_{y}} \in H_{n}(X,X-L_{y})$
and $\beta_{y}\in H_{n}(Y,Y-y)$ be the orientation classes along
$L_{y}$ at $y$; respectively. There exists an integer number
$\deg(f,y)$ such that $f_{*}(\alpha_{L_{y}})={\rm
deg}(f,y)\cdot\beta_{y}$, where $f_{*}:H_{n}(X,X-L_{y})\to
H_{n}(Y,Y-\{y\})$ is the homomorphism induced by $f$. The number
$\deg(f,y)$ is called {\it degree of $f$ at $y$.}
\end{definition}

\begin{definition}\label{def:grau}\rm Let $f:X\to Y$ be a
{\it discrete map} at a point $x_{0}$ and let us denote by
$y_{0}=f(x_{0})$. Consider $\alpha_{x_{0}}\in H_{n}(V,V-x_{0})$ and
$\beta_{y_{0}}\in H_{n}(Y,Y-y_{0})$ the orientation classes at
$x_{0}$ and $y_{0}$, respectively. There exists an integer number
$\deg(f,x_{0})$ such that $f_{*}(\alpha_{x_{0}})={\rm
deg}(f,x_{0})\cdot\beta_{y_{0}}$, where the homomorphism
$f_{*}:H_{n}(V,V-x_{0})\to H_{n}(Y,Y-y_{0})$ is induced by $f$. The
number $\deg(f,x_{0})$ is called {\it local degree of $f$ at
$x_{0}$.}
\end{definition}

\begin{definition}\label{def:graugeneralizado}\rm
Suppose that $f:X\to Y$ is not necessarily a  discrete map. Define
\begin{equation}\deg(f,x)=\left\{ \begin{array}{ll}
0, & \mbox{if $f$ is not {\it discrete} at $x.$\hspace{.3cm}}  \\
 \deg(f,x), &  \mbox{\rm if $f$ is {\it discrete} at $x$, in
 the sense of Definition \ref{def:grau}.\nonumber}
\end{array} \right.
\end{equation}
\end{definition}

 The proof of the following two propositions can be found in \cite{dold}.

\begin{proposition}\label{degree:prop} Let $X$ and $Y$ be oriented connected
topological manifolds of dimen\-sion $n\geq 1$ and let $f:X\to Y$ be
a continuous map. Consider $y\in Y$  such that $f^{-1}(y_{0})$ is
finite. Then,
\begin{equation} \deg(f,y)=\sum_{x\in f^{-1}(y)}\deg(f,x).\nonumber
\end{equation}
\end{proposition}

\begin{proposition} \label{degree:prop1} Let $X$ and $Y$ be oriented connected
topological manifolds of dimension $n\geq 1$ and $f:X\to Y$ be a
continuous map. Let us consider a compact connected subset $K$ of
$Y$ such that $L_{K}=f^{-1}(K)$ is compact. Then, $\deg(f,y)$ does
not depend of $y\in K$.
\end{proposition}

 \noindent As an immediate consequence of
Proposition \ref{degree:prop1} we have the following:

\begin{corollary} Let $X$ and $Y$ be oriented connected
topological manifolds of dimension $n\geq 1$ and $f:X\to Y$ be a
proper continuous map. Then, for all $y\in Y,$ \, {\rm
deg}$(f,y)=\deg(f)$.
\end{corollary}

 In \cite{Vaisala}, Väisälä proved  the following version of the $\breve{C}$ernavskii's
theorem ( see \cite{cerna, cerna1}):

\begin{theorem}\label{cerna} Let $X$ and $Y$ be
topological manifold of dimension $n$. Suppose that $f:X\to Y$ is an
open and discrete map. Then $\dim (f(B_{f}))\leq n-2$ and $\dim
(B_{f})\leq n-2$, where $B_{f}$ denotes the set of points $x$ of $X$
such that $f$ is not a local homeomorphism at $x$.
\end{theorem}

\section{Persistence of the sign of $\deg(f,x)$}

The following lemma will be fundamental in the proof of the versions
of the Implicit Function Theorem and  Darboux Theorem.

\begin{lemma}[Key Lemma] \label{keylemma} Let $X$ and $Y$ be oriented
connected topological manifolds of dimension $n$. Suppose that
$f:X\to Y$ is an open and discrete map. Then, for any $x\in X,$ one
has that $\deg(f,x) \ne 0;$ moreover, $\deg(f,x)$ has always the
same sign.
\end{lemma}

\begin{proof} It follows from Theorem \ref{cerna} that
$\dim B_{f}\leq n-2$, then $X-B_{f}$ is connected. Since $f$ is a
local homeomorphism on $X-B_{f}$, one has that
\begin{equation}\label{f1}\deg(f,x)=c, \forall x\in X-B_{f},
\mbox{where either $c=1$ or $c=-1$}.
\end{equation}
 Let $x_{0}\in B_{f}$ and  denote by $y_{0}=f(x_{0})$. Since $f$ is
a discrete map and $X$ is locally compact, we can choose an open
neighborhood $V$ of $x_{0}$ such that $\overline{V}$ is compact and
$f(x)\not= f(x_{0})$ for all $x\in \overline{V}- x_{0}.$  Then,
there exists an open neighborhood $W$ of $y_{0}=f(x_{0})$ such that
$\overline{W}$ is a compact and connected subset of $Y$ and
$(f|_{\overline V})^{-1}(\overline{W})=
(f|_{V})^{-1}(\overline{W})\subset V$. Therefore,
$(f|_{V})^{-1}(\overline{W})$ is compact and it follows from
Proposition \ref{degree:prop1} that
\begin{equation}\label{f2}
\deg(f|_{V},y)=\deg(f|_{V},y_{0}), \forall y\in \overline{W}
\end{equation}
On the other hand, since $U=(f|_{V})^{-1}(W)$ is an open set in $X$
and $f$ is a open map, we have that $f|_{V}(U)\subset \overline{W}$
is an open set in $Y$ and since $Y-f(B_{f})$ is dense in $Y$ (recall
that  $\dim (f(B_{f}))\leq n-2$ by Theorem \ref{cerna}), there
exists $y_{1}\in f|_{V}(U)-f(B_{f})$. Therefore, it follows from
(\ref{f2}) that $\deg(f|_{V},y_{1})=\deg(f|_{V},y_{0})$. Let
$f^{-1}(y_{1})=\{ x_{1}, \ldots , x_{k}\}$ in $X-B_{f}$. Then, by
Proposition \ref{degree:prop} we have that
\begin{equation}\label{f3}
\deg(f,x_{0})=\deg(f|_{V},y_{0})= \deg(f|_{V},y_{1})=\sum_{x_{i}\in
f^{-1}(y_{1})} \deg(f|_{V},x_{i}).
\end{equation}
On the other hand, since $x_{i}\in X- B_{f}$ for $i=1,\ldots,k$, by
formula (\ref{f1}), $\deg(f,x_{i})=c$ where either $c=1$ or $c=-1$.
Then, it follows from (\ref{f3}) that $\deg(f,x_{0})=kc$ and
therefore $\deg(f,x)$ has always the same sign, for any $x\in X$.
\end{proof}

\section{The Implicit Function Theorem}

  In this section we shall show a new
homological version of the Implicit Function Theorem.

\medskip

\begin{theorem}[Implicit Function Theorem] \label{IFT} Let $X$
be a locally pathwise connected Hausdorff space and let $Y,Z$ be
oriented connected topological manifolds of dimension $n$. Let $f:X
\times \, Y \,\to\, Z$ be a continuous map such that, for all $x \in
X$, the map $f_x:Y \to Z$ given by $f_x(y)=f(x,y)$ is open and
discrete. Suppose that for some $(x_0,y_0) \in X \times \, Y$,
$|\deg(f_{x_0},y_0)|=1$. Then there exists a neighborhood $V$ of
$x_0$ in $X$ and a continuous function $g:V \to Y$ such that
$f(x,g(x))=f(x_0,y_0)$, for all $x \in V$.
\end{theorem}

\begin{proof} Let $z_0=f(x_0,y_0)$. Since $Y$ is locally compact and $f_{x_{0}}
$ is a discrete map, we can choose a compact neighborhood $W\subset
Y$ of $y_{0}\in Y$ satisfying
\begin{eqnarray} (f_{x_{0}})^{-1}(z_{0})\cap W= y_{0}. \end{eqnarray}
We will first show that  {\it for any compact neighborhood $K\subset
{\mbox{\rm Int}}(W)$ containing $y_{0}$, there exists a neighborhood
$V$ of $x_{0}$ such that}
\begin{eqnarray} \label{step1}
({{f_{x}}_{|}}_{W})^{-1}(\{z_{0}\})\subseteq K, \hspace{1cm}\forall x\in V.
\end{eqnarray}
\noindent In fact, suppose that for each neighborhood $V$ of $x_{0}$
there exists  $(x_{\mbox{\tiny $V$}},y_{\mbox{\tiny $V$}})$ in $V
\times (W-K)$ such that $f(x_{\mbox{\tiny $V$}},y_{\mbox{\tiny
$V$}})=z_{0}$. Let us consider a genera\-lized sequence, called a
net, $\left((x_{\mbox{\tiny $V$}},y_{\mbox{\tiny
$V$}})\right)_{\mbox{\tiny {V}}\in \mathcal{V}}$, where
$\mathcal{V}$ is the collection of all the neighborhoods of $x_{0}$,
partially ordered by reverse inclusion (for details see \cite[pg.187
and 188]{Munkres}). Therefore, lim $x_{\mbox{\tiny $V$}}=x_{0}$ and
since $(y_{\mbox{\tiny $V$}})$ is a net contained in the compact
subset $W-\mbox{\rm Int}(K)$, there exists $y_{1}\in W-\mbox{\rm
Int}(K)$ which is a limit point of some convergent subnet of
$(y_{\mbox{\tiny $V$}})$. Hence, $(x_{0},y_{1})$ is a limit point of
some subnet of $(x_{\mbox{\tiny $V$}},y_{\mbox{\tiny $V$}})$. Since
$f$ is a continuous map, one has $f(x_{0},y_{1})=z_{0}$ which
implies that $y_{1}=y_{0}$, contradicting the fact that $y_{1}\in
W-\mbox{\rm Int}(K)$. Therefore, there exists a neighborhood $V$ of
$x_{0}$ satisfying (\ref{step1}).

Choose a compact neighborhood $K\subset {\mbox{\rm Int}}(W)$ of
$y_{0}$. It follows from (\ref{step1}) that for each $x\in V$, the
map of pairs $f_{x}:(W,W-K)\to (Z,Z-z_{0})$ is well defined. Since
$X$ is locally pathwise connected, we can assume that $V$ is a
pathwise connected neighborhood of $x_{0}$ and therefore, for each
$x$ in $V$, there exists a path $\alpha $ in $V$ joining $x_{0}$ to
$x$. We define the homotopy  of pairs $H:(I\times W,I\times
(W-K))\to (Z,Z-z_{0})$ given by
\begin{eqnarray}\label{homotopy}
H(t,y)=f(\alpha(t),y)=f_{\alpha(t)}(y), \hspace{1cm} \forall (t,y)\in I\times W.
\end{eqnarray}

Since $W$ is compact and $H$ is a proper homotopy between
$f_{x_{0}}|_{W}$ and $f_{x}|_{W}$, we obtain
$\deg(f_{x}|_{W},z_{0})=\deg(f_{x_{0}}|_{W},z_{0})$. Now, as
$\deg(f_{x_{0}}|_{W},z_{0})= \deg(f_{x_{0}},y_{0})$ and as, by
hypotheses, $|\deg(f_{x_{0}},y_{0})|=1$ we obtain
\begin{eqnarray}\label{deg1}|\deg(f_{x}|_{W},z_{0})|=|\deg(f_{x_{0}}|_{W},z_{0})|
=|\deg(f_{x_{0}},y_{0})|=1.\end{eqnarray}

\noindent Since  $f_{x}$ is open and discrete, it follows from Lemma
\ref{keylemma} that $\deg(f_{x},y)\not= 0$, for any $y\in Y$, and
$\deg(f_{x},y)$  has always the same sign. Thus, if
$(f_{x}|_{W})^{-1}(z_{0})=\{y_{1}, \cdots, y_{k}\}\subset W$ with
$k\geq 2$, we have that
\begin{eqnarray}
 \hspace{0.2cm}  |\deg(f_{x}|_{W},z_{0})|=|\sum_{i=1}^{k}\deg(f_{x},y_{i})|=\sum_{i=1}^{k}|
\deg(f_{x},y_{i})|>1,
\end{eqnarray}
which contradicts (\ref{deg1}). Therefore, for each $x\in V$  there
exists a unique $y\in K\subset W$ such that
$(f_{x}|_{W})^{-1}(z_{0})=y$; in other words, for each $x\in V$
there exists a unique $y=g(x)\in K$ such that
$f_{x}(g(x))=f(x,g(x))=z_{0}$ for each $x\in V$.

We will show that the map $g:V\to K\subset Y$ is continuous. Let $A$
be a neighborhood  of $y=g(x)$ such that $A\subset K$. Let us assume
that for any neighborhood $U$ of $x$, there exists $x_{\mbox{\tiny
$U$}}$ in $U$ such that $g(x_{\mbox{\tiny $U$}})$ belongs to the
compact subset  $K-A$. Let $\bar{y}$ in $K-A$ be the limit point of
the some convergent subnet of $(g(x_{\mbox{\tiny $U$}}))$. Thus,
$(x,\bar{y})$ is the limit point of the some convergent subnet of
$(x_{\mbox{\tiny $U$}},g(x_{\mbox{\tiny $U$}}))$. Since $f$ is
continuous and $g$ is given implicitly by the equation
$f(x_{\mbox{\tiny $U$}},g(x_{\mbox{\tiny $U$}}))=z_{0}$ we have that
$f(x,\bar{y})=z_{0}$, which implies that $y=\bar{y}$, contradicting
the fact that $\overline{y}\in K-A$.
\end{proof}

\begin{corollary}\label{coro-dif}   Let $X$ be a locally pathwise
connected Hausdorff space. Let $U\subset \mathbb{R}^{n}$ be an open
subset of $\mathbb{R}^{n}$ and let $f:X\times U\to \mathbb{R}^{n}$
be a continuous map. Suppose that $f_{x}:U\to \mathbb{R}^{n}$ is a
differentiable (not necessarily $C^{1}$) map without critical
points, for each $x\in X$. Then there exist a neighborhood $V$ of
$x_{0}$ and a continuous function $g:V\to U$ such that
$f(x,g(x))=z_{0}$, for each $x\in V$.

\end{corollary}

\begin{proof} Since $f_{x}:U\to \mathbb{R}^{n}$ is a
differentiable (not necessarily $C^{1}$) map without critical
points, we have that  $f_{x}$ is an open and {\it discrete} map, for
each $x$ in $X$. Therefore, the result follows from Theorem
\ref{IFT}
\end{proof}

\begin{corollary}\label{coro-eq-dif} Let $I \times V
\times W$ be an open neighborhood of $(t_0,x_0,y_0)$ in $\R \times
\R^n \times \R^m$ and let $F:(I \times V \times W, (t_0,x_0,y_0))
\longmapsto (\R^m,0)$ be a continuous map. Suppose that, for all
$(t,x) \in I \times V$, the map $y \in W \longrightarrow F(t,x,y)$
is differentiable (but not necessarily $C^1$) and without critical
points. Then the differential equation
$$F(t,x,x')=0, \quad x(t_0)=x_0, \quad x'(t_0)=y_0$$
has a solution in some interval $(t_0-\varepsilon,t_0+\varepsilon)$.
\end{corollary}

\begin{proof} By Corollary \ref{coro-dif}, we have that there
exists a neighborhood $(t_0-\varepsilon_1,t_0+\varepsilon_1) \times
V_1$ of $(t_0,x_0)$ and a function
$g:(t_0-\varepsilon_1,t_0+\varepsilon_1) \times V_1 \longrightarrow
W$ such that
$$F(t,x,g(t,x))=0$$
By Peano Theorem, there exists a solution
$\varphi:(t_0-\varepsilon,t_0+\varepsilon)$ of the differential
equation
$$x'=g(t,x), \quad x(t_0)=x_0, \quad x'(t_0)=y_0.$$
This implies the Corollary.
\end{proof}

\section{Generalizations of the Darboux theorem}

 The classical Darboux theorem states that if $f:[a,b]\to
\mathbb{R}$ is a differentiable map which has the property that
$f^{\prime}(a)< 0$ and $f^{\prime}(b)>0$, then there exists a point
$c\in (a,b)$ such that $f^{\prime}(c)=0$.

 In order to obtain some versions of the Darboux theorem we apply
the results previously obtained.

\begin{theorem}[A homological version of the Darboux theorem]
\label{HDT} Let $M$ and $N$ be oriented connected topological
manifolds of dimension $n$ and let $f:M\to N$ be a continuous map.
Suppose that there exist $x_{0}$ and $x_{1}$ in $M$ such that
$\deg(f,x_{0})<0$ and $\deg(f,x_{1})>0$, then there exists $x_{2}$
in $M$ such that $\deg(f,x_{2})=0$
\end{theorem}

\begin{proof} If $f$ is not a {\it discrete map} at some $x\in M$,
it follows from Definition \ref{def:graugeneralizado} that
$\deg(f,x)=0$. In this way, suppose that $f$ is a {\it discrete map}
at $x$, for every $x\in M$. Therefore, if $\deg(f,x)\not= 0$ for any
$x\in M$, then $f$ is also an open map and from Lemma \ref{keylemma}
we conclude that $\deg(f,x)$ has always the same sign, for any $x\in
M$, which is a contradiction.\end{proof}

 The following theorems are differentiable versions of the Darboux
theorem. Let us observe that when $f$ is of class $C^{1},$ these
results are trivial.

\medskip

\begin{corollary}\label{DDT} Let $M$ and $N$ be oriented
connected topological manifolds of dimension $n$ and let $f:M\to N$
be a differentiable map. Suppose that there exist $x_{0}$ and
$x_{1}$ in $M$ such that $\det[f^{\prime}(x_{0})]<0$ and
$\det[f^{\prime}(x_{1})]>0$. Then $f$ has a critical point.
\end{corollary}

\begin{proof} Suppose that $f$ does not have critical points. In
this case, one has that for any $x\in M$, $|\deg(f,x)|=1\not= 0$.
Thus, $f$ is an open and {\it discrete} map and it follows from
Lemma \ref{keylemma} that $\deg(f,x)$ has always the same sign,
which is a contradiction.
\end{proof}

\begin{corollary} \label{DDT1}Consider $f,g:M\to \mathbb{R}^{n}$
differentiable maps, where $M$ is an oriented connected topological
manifold of dimension $n$. Suppose that there exist $\alpha\in
\mathbb{R}$ and $x_{0},x_{1}\in M^{n}$ such that
\begin{equation} \det[f^{\prime}(x_{0})-\alpha
g^{\prime}(x_{0})]<0 \hspace{.4cm}\mbox{and}\hspace{.4cm} \det
[f^{\prime}(x_{1})-\alpha g^{\prime}(x_{1})]>0.\nonumber
\end{equation}
Then, there exists $x_{2}\in M$ such that $\det[f^{\prime}(x_{2})-
\alpha g^{\prime}(x_{2})]$ is equal to zero.
\end{corollary}

\begin{proof} It suffices to apply Theorem \ref{DDT} for the map
$h=f-\alpha g$.
\end{proof}

 As a direct consequence of Theorem \ref{DDT1} one has the
following version of the classical Darboux theorem for
differentiable maps from $\mathbb{R}^{n}$ into $\mathbb{R}^{n}$.

\begin{corollary}{\bf (Differentiable Darboux theorem)}
\label{DDT2} Let $f:U\to \mathbb{R}^{n}$ be a differentiable map,
where $U$ is an open connected subset of $\mathbb{R}^{n}$. Suppose
that there exist $\alpha\in \mathbb{R}$ and $x_{0},x_{1}\in U$ such
that $\det[f^{\prime}(x_{0})-\alpha
 I]<0$ and $\det[f^{\prime}(x_{1})-\alpha I]>0$. Then, there exists $x_{2}\in U$
such that $\det[f^{\prime}(x_{2})-\alpha I]$ is equal to zero $($
i.e. $\alpha$ is an eigenvalue of $f^{\prime}(x_{2}) )$.\qed
\end{corollary}

 Now consider $f$ and $U$ under the same assumptions of Corollary
\ref{DDT2} and let us denote by
$p_{0}(\lambda)=\det[f^{\prime}(x_{0}))-\lambda I]$ and by
$p_{1}(\lambda)=\det[f^{\prime}(x_{1}))-\lambda I]$. Let $n_{0}$ and
$n_{1}$ natural numbers. In these conditions, we prove the following

\begin{corollary} \label{DDT3}Let $x_{0},x_{1}\in U$ and $\alpha \in
\mathbb{R}$ such that $p_{0}(\lambda)=q_{0}(\lambda)\cdot (\alpha -
\lambda)^{n_{0}}$ and $p_{1}(\lambda)=q_{1}(\lambda)(\alpha
-\lambda)^{n_{1}}$, where $q_{0}(\lambda)$ and $q_{1}(\lambda)$ are
not null polynomials. Suppose that $n_{0}$ is odd and $n_{1}$ is
even. If $q_{0}(\alpha)q_{1}(\alpha)>0$ then there exists $\delta>0$
satisfying the following condition: for each $\lambda \in
(\alpha,\alpha+\delta)$ there exists $x_{2}=x_{2}(\lambda)$ such
that $\det[f^{\prime}(x_{2})-\lambda I]$ is equal to
zero.\end{corollary}

\begin{proof} Since $n_{0}$ is odd and $n_{1}$ is even, one has
that, for $\lambda > \alpha$ close to $\alpha$, the polynomials
$p_{0}(\lambda)$ and $p_{1}(\lambda)$ have different signs. It
follows from Corollary \ref{DDT2} that if $\delta > 0$ is small
enough and $\lambda \in (\alpha,\alpha + \delta)$, there exists
$x_{2}=x_{2}(\lambda)$ such that
$p_{2}(\lambda)=\det[f^{\prime}(x_{2})-\lambda I]$ is equal to zero.
\end{proof}

\begin{remark} Theorem \ref{DDT3} remains the same in the case that
$q_{0}(\alpha)q_{1}(\alpha)<0$ and $\lambda \in (\alpha-\delta,\alpha)$.
\end{remark}

\end{document}